\documentclass[11pt]{article}

\usepackage{amsmath, amsthm, amssymb, tikz}
\usepackage{amsfonts}
\usepackage{fullpage}
\usepackage[enableskew]{youngtab}

\newcommand\beq{\begin{equation}}
\newcommand\eeq{\end{equation}}
\newcommand\bce{\begin{center}}
\newcommand\ece{\end{center}}
\newcommand\bea{\begin{eqnarray}}
\newcommand\eea{\end{eqnarray}}
\newcommand\ba{\begin{array}}
\newcommand\ea{\end{array}}
\newcommand\ben{\begin{enumerate}}
\newcommand\een{\end{enumerate}}
\newcommand\bit{\begin{itemize}}
\newcommand\eit{\end{itemize}}
\newcommand\brr{\begin{array}}
\newcommand\err{\end{array}}
\newcommand\bt{\begin{tabular}}
\newcommand\et{\end{tabular}}

\newcommand{\la}{\lambda}

\newcommand{\NN}{\mathbb{N}}

\newtheorem{theorem}{Theorem}[section]

\newtheorem{lemma}[theorem]{Lemma}

\numberwithin{figure}{section}

\title{On exponential growth of degrees}

\author{Yuval
Roichman~\thanks{Department of Mathematics, Bar-Ilan University,
 Ramat-Gan 52900, Israel.  {\tt yuvalr@math.biu.ac.il}.}}

\begin{document}

\maketitle

\begin{abstract}
A short proof to a recent theorem of Giambruno and Mishchenko is
given in this note.
\end{abstract}

\section{The theorem}

The following theorem was recently proved by Giambruno and
Mishchenko.

\begin{theorem}\label{main}\cite[Theorem 1]{GM}
For every $0<\alpha<1$, there exist $\beta>1$ and $n_0\in \NN$,
such that for every partition $\la$ of $n\ge n_0$ with
$\max\{\la_1,\la'_1\}<\alpha n$
\[
f^\la\ge \beta^n .
\]
\end{theorem}

The proof of Giambruno and Mishchenko is rather complicated and
applies a clever order on the cells of the Young diagram. It
should be noted that Theorem~\ref{main} is an immediate
consequence of Rasala's lower bounds on minimal
degrees~\cite[Theorems F and H]{Rasala}. The proof of Rasala is
very different and not less complicated; it relies heavily on his
theory of degree polynomials. In this short note we suggest a
short and relatively simple proof to Theorem~\ref{main}.

\bigskip

%Here we suggest a short proof of the theorem.
First, note that the following weak version is an immediate
consequence of the hook-length formula.

\begin{lemma}\label{lemma0}
The theorem holds for every $0<\alpha< \frac{1}{2e}$.
\end{lemma}

\begin{proof}
Under the assumption, for every $(i,j)\in [\la]$
\[
h_{i,j}\le h_{1,1}\le \la_1+\la_1'\le 2\alpha {n}.
\]
Hence, by the hook formula together with Stirling formula, for
sufficiently large $n$
\[
f^\la={n!\over \prod\limits_{(i,j)\in [\la]}h_{i,j}}\ge {n!\over
({2\alpha n})^n}\ge {({n\over e})^n\over ({2\alpha n})^n}
%=({1\over 2e\alpha})^n
=\beta^n,
\]
where, by assumption, $\beta:={1\over 2e\alpha}>1$.
\end{proof}

\section{Two lemmas}

\begin{lemma}\label{lemma1}
For every $\la\vdash n$
\[
\prod\limits_{(i,j)\in [\la]\atop 1<i}h_{ij}\le (n-\la_1)!
\]
\end{lemma}

\begin{proof}
For $\la=(\la_1,\la_2,\dots,\la_t)\vdash n$ let
$\bar\la:=(\la_2,\dots,\la_t)\vdash n-\la_1$. Then
\[
1\le f^{\bar\la}={(n-\la_1)!\over \prod\limits_{(i,j)\in
[\la]\atop 1<i}h_{ij}}.
\]
\end{proof}

\begin{lemma}\label{lemma2}
For every $\la\vdash n$ and $1\le k\le \la_1$
\[
\prod\limits_{(1,j)\in [\la]}h_{1j} \le {n\choose k}
(\la_1+\lfloor n-{\la_1\over k}\rfloor)! .
\]
\end{lemma}

\begin{proof}
%For every $1\le j<\la_1$, $h_{1,j}>h_{1,j+1}$.
Obviously, $h_{1,1}>h_{1,2}>\cdots>h_{1,\la_1}$. Since $h_{1,1}\le
n$ it follows that
\[
\prod\limits_{(1,j)\in [\la]\atop j\le k}h_{1j}\le (n)_k
\]
and
\[
\prod\limits_{(1,j)\in [\la]\atop k<j}h_{1j}\le
(h_{1,k})_{\la_1-k}.
\]
To complete the proof, notice, that by definition,
$\sum\limits_{i=1}^k \bar\la_i'\le n-\la_1$. Hence $\bar\la'_k\le
\lfloor {n-\la_1\over k}\rfloor$ and thus
\[
h_{1,k}=\la_1-k+\la'_k=\la_1-k+1+\bar\la_k'\le \la_1+\bar\la_k'\le
\la_1+\lfloor {n-\la_1\over k}\rfloor.
\]
We conclude that
\[
\prod\limits_{(1,j)\in [\la]\atop k<j}h_{1j}\le
(h_{1,k})_{\la_1-k}\le (\la_1+\lfloor{n-\la_1\over
k}\rfloor)_{\la_1-k}={(\la_1+\lfloor{n-\la_1\over k}\rfloor)!\over
(\lfloor{n-\la_1\over k}\rfloor+k)!}\le
{(\la_1+\lfloor{n-\la_1\over k}\rfloor)!\over k!}.
\]
Thus
\[
\prod\limits_{(1,j)\in [\la]}h_{1j}=\prod\limits_{(1,j)\in
[\la]\atop j\le k}h_{1j}\prod\limits_{(1,j)\in [\la]\atop
k<j}h_{1j} \le (n)_k {\la_1+\lfloor{n-\la_1\over k}\rfloor!\over
k!}= {n\choose k} (\la_1+\lfloor{n-\la_1\over k}\rfloor)!.
\]

\end{proof}

\section{Proof of Theorem~\ref{main}}

For the sake of simplicity the floor notation is omitted in this
section.

\bigskip

By Lemmas~\ref{lemma1} and~\ref{lemma2},
\[
f^\la={n!\over \prod\limits_{(i,j)\in [\la]}h_{ij}}={n!\over
\prod\limits_{(1,j)\in [\la]}h_{1j}\prod\limits_{(i,j)\in
[\la]\atop 1<i}h_{ij}}\ge {n!\over (n-\la_1)!{n\choose k} (\la_1+
n-{\la_1\over k})!}= {(n-k)!k!\over (n-\la_1)!
(\la_1+{n-\la_1\over k})!}.
\]
Denote $\gamma_n:={\la_1\over n}$. By Lemma~\ref{lemma0}, we may
assume that ${1\over 2e} < \gamma_n <\alpha$. Choose $k=\epsilon
n$ for a constant $\epsilon=\epsilon(\alpha)$ to be defined later.
By the Stirling formula, the lower bound in the RHS asymptotically
equals to
\[
{((1-\epsilon)n)!(\epsilon n)! \over ((1-\gamma_n)n)! (\gamma_n
n+{1-\gamma_n\over \epsilon})!}\sim
\sqrt{\epsilon(1-\epsilon)\over (1-\gamma_n)(\gamma_n+{c_n\over
n})} \cdot {(1-\epsilon)^{(1-\epsilon)n}\epsilon^{\epsilon n}
\over (1-\gamma_n)^{(1-\gamma_n)n} (\gamma_n+{c_n\over
n})^{(\gamma_n+{c_n\over n})n}}\cdot ({e\over n})^{c_n} ,
\]
where $c_n:={1-\gamma_n\over \epsilon}$. Thus ${1-\alpha\over
\epsilon}<c_n<{2e-1\over 2e \epsilon}$.
%\bigskip
Hence, for sufficiently large $n$ %and appropriate $\epsilon$
\[
\lim_{n \rightarrow \infty} \inf (f^\la)^{1/n}\ge \lim_{n
\rightarrow \infty} \inf \left(\sqrt{\epsilon(1-\epsilon)\over
(1-\gamma_n)(\gamma_n+{c_n\over n})} \cdot
{(1-\epsilon)^{(1-\epsilon)n}\epsilon^{\epsilon n} \over
(1-\gamma_n)^{(1-\gamma_n)n} (\gamma_n+{c_n\over
n})^{(\gamma_n+{c_n\over n})n}}\cdot ({e\over
n})^{c_n}\right)^{1/n}
\]
\[
\ge \min_{\gamma\in [{1\over 2e},\alpha]}
{\epsilon^\epsilon(1-\epsilon)^{1-\epsilon}\over \gamma_n^{\gamma}
(1-\gamma)^{1-\gamma} }.
\]
The function $f(x):=x^x (1-x)^{1-x}$ is differentiable in the open
interval $(0,1)$, symmetric around its minimum at $x={1\over 2}$,
decreasing in $(0,{1\over 2}]$, increasing in $[{1\over 2}, 1)$,
strictly less than 1 in this interval and tends to 1 at the
boundaries.  Since $\gamma_n\in [{1\over 2e},\alpha]\subseteq
(0,1)$ it follows that $f(\gamma_n)\le f(\alpha)$ if $1-{1\over
2e}\le \alpha$ and $\le f({1\over 2e})$ otherwise.
%%\[
%%\gamma_n^{\gamma_n}
%%(1-\gamma_n)^{1-\gamma_n}\le \alpha^\alpha (1-\alpha)^{1-\alpha}
%%\]
%\[
%f(\gamma_n)\le \max\{f({1\over 2e}), f(\alpha)\}
%\]
%Hence,
Choosing $\epsilon=\epsilon(\alpha)$ such that $\epsilon\le
\delta\min\{1-\alpha,{1\over 2e}\}$ for some very small $\delta>0$
we conclude that
\[
\lim_{n \rightarrow \infty} \inf (f^\la)^{1/n}\ge \min_{\gamma\in
[{1\over
2e},\alpha]}{\epsilon^\epsilon(1-\epsilon)^{1-\epsilon}\over
\gamma^{\gamma} (1-\gamma)^{1-\gamma} }\ge \min\{ {f(\delta{1\over
2e})\over f({1\over 2e})},  {f(\delta (1-\alpha))\over
f(1-\alpha)}\}>1,
\]
completing the proof.

\qed

%\bigskip

%\begin{remark}
%A careful analysis of the above proof shows that
%$(f^\la)^{1/n}>\beta$ for every $\beta<{1\over \alpha^\alpha
%(1-\alpha)^{1-\alpha}}$. This slightly improves the bound
%$\beta<{1\over \alpha}$ of~\cite{GM}.
%\end{remark}

\bigskip

\noindent{\bf Acknowledgements.} Thanks to Amitai Regev for
fruitful discussions and references.

\end{document}